\documentclass[a4paper,reqno,12pt]{amsart}
\usepackage{amsmath}
\usepackage{amssymb}
\usepackage{amscd}
\usepackage{fullpage}
\usepackage{graphicx}
\usepackage[small,heads=vee,nohug]{diagrams}
\diagramstyle{labelstyle=\scriptstyle}
\newcommand{\abs}[1]{\left\vert #1 \right\vert}
\newcommand{\floor}[1]{\left\lfloor #1 \right\rfloor}

\newcommand{\Spinc}{\text{Spin}^c}
\newcommand{\Spin}{\text{Spin}}
\DeclareMathOperator{\PD}{PD}

\DeclareMathOperator{\Char}{Char}
\DeclareMathOperator{\Hom}{Hom}
\DeclareMathOperator{\id}{id}
\theoremstyle{plain}
\newtheorem{thm}{Theorem}[section]
\newtheorem{defn}[thm]{Definition}
\newtheorem{lem}[thm]{Lemma}
\newtheorem{prop}[thm]{Proposition}
\newtheorem{cor}[thm]{Corollary}

\title{Deficiency Symmetries of Surgeries in $S^3$}
\author{Julian Gibbons}
\address{Department of Mathematics, Imperial College London}
\email{j.gibbons09@imperial.ac.uk}

\begin{document}

\begin{abstract}
We examine certain symmetries in the deficiencies of a rational surgery on a knot in $S^3$ by comparing the $\Spinc$-structures on the rational surgery with those on a related integral surgery. We then provide an application of these symmetries in the form of a theorem that obstructs Dehn surgeries in $S^3$. This last part unifies and generalises theorems by Greene underlying his work on the cabling conjecture, the lens space realisation problem, and the unknotting number of alternating 3-braids.
\end{abstract}

\maketitle

\section{Introduction}
\label{introduction}

In \cite{GreeneGenus}, Greene laid the basics for the follow beautiful theorem. He proved that if a lens space $L(p^\prime,q^\prime)$ (using the convention of $-p^\prime/q^\prime$-surgery on the unknot) is obtained by $p^\prime$-surgery on a knot $C \subset S^3$, then there exists an integral matrix $A$ such that
$$-AA^t = Q_X \oplus (-p^\prime),$$
where
\begin{enumerate}
\item $Q_X$ is the adjacency matrix of the linear graph with weights $-b_i$ appearing in the Hirzebruch-Jung continued fraction $p^\prime/q^\prime = [b_1,\dots,b_\ell]$ (i.e. $Q_X$ represents the intersection form of the graph's corresponding plumbed 4-manifold $X$); and
\item The entries of the final row of $A$ form a \emph{changemaker set} (i.e. a collection of non-negative integers $\sigma_i$ with the property that, given coins worth $\sigma_i$, one can make up any value from zero to their sum).
\end{enumerate}
By combining this theorem with some ingenious combinatorics, Greene was able to achieve spectacular success resolving the long-standing lens space realisation problem \cite{GreeneLens}. A similar theorem, involving the double branched cover of an alternating 3-braid $K$ and half-integral surgeries, also allowed him to classify the $K$ of this type with unknotting number one \cite{GreeneBraid}.

The main idea in this paper is to take the two ``changemaker'' theorems above and unify them in one. To this end, we will largely be concerned with the \emph{deficiencies} of a $-p/q$-surgery on a knot $C$ in $S^3$, where $p , q > 0$ are coprime. These objects are defined as the differences
$$D^{p/q}_C(\mathfrak{t}) := d(S^3_{-p/q}(C),\mathfrak{t}) - d(S^3_{-p/q}(U),\mathfrak{t}),$$
where $U$ is the unknot, $\mathfrak{t}$ a $\Spinc$-structure, and $d$ the correction term of Ozsv\'ath and Szab\'o (see \cite{OSAbsolute}). Since this implicitly requires a bijection
$$\Spinc(S^3_{-p/q}(C)) \longleftrightarrow \Spinc(S^3_{-p/q}(U)),$$
we stipulate that the one used here is the standard one from the literature \cite{GreeneBraid, GreeneGenus, GreeneLens, OSUnknot}. When it is clear what we mean, we may drop the $C$ from $D^{p/q}_C$.

In order to prove our result, we begin by establishing certain symmetries among the $D^{p/q}_C(i)$. Explicitly, we have the following theorem, where $p = nq - r$ and $0 < r < q$.

\begin{thm}
\label{main}
Let $C$ be a knot in $S^3$, and let $p,q > 0$ be coprime. Then there is a function $\mathfrak{r} : \Spinc(S^3_{-p/q}(C)) \rightarrow \Spinc(S^3_{-n}(C))$ such that the following diagram commutes:
\begin{diagram}
\Spinc(S^3_{-p/q}(C)) &\rTo^{\phantom{PPP}\mathfrak{r}\phantom{PPP}} & \Spinc(S^3_{-n}(C))\\
&\rdTo_{D^{p/q}} &\dTo^{D^n} \\
& & \mathbb{Q}
\end{diagram}
and the fibres of $\mathfrak{r}$ are of size $q$, with one exception over an element of $\Spinc(S^3_{-n}(C))$ which minimises the value of $D^n$. In particular, conjugation on $\Spinc(S^3_{-n}(C))$, under which $D^n$ is invariant, lifts to a function on $\Spinc(S^3_{-p/q}(C))$ under which $D^{p/q}$ is invariant.
\end{thm}

Stated thus, the theorem is in fact not difficult to prove, though the exhibition of an $\mathfrak{r}$ requires considerably more effort. We provide one example towards the middle of the paper, before using it to generalise the two changemaker theorems. We let $p/q = [a_1,\dots,a_\ell]$, where $a_i \geq 2$ for $i \geq 2$ and $n = a_1$. It is not difficult to prove that such an expansion always exists.

\begin{thm}
\label{changemaker}
Suppose that $Y = S^3_{-p/q}(C)$ for some knot $C \subset S^3$ and coprime $p,q>0$, that $W$ is the trace of the corresponding integral surgery in Figure \ref{trace}, and that $-Y$ bounds a sharp, simply connected, negative-definite smooth 4-manifold $X$ with intersection form $Q_X$ and free $H_2(X)$. Then if
$$d(Y,i) - d(S^3_{-p/q}(U),i) = 0$$
for either (a) one value of $i$ if $n$ is odd; or (b) $q-r+1$ values of $i$ if $n$ is even, there exists an integral matrix $A$ such that
$$-AA^t = Q_X\oplus Q_W.$$
In addition, if $q \neq 1$, one can choose $A$ so that its last $\ell$ rows have the form
$$\left(\begin{array}{ccccccccccccccc}
\sigma_r &\dots &\sigma_1 &1 &0 &      &  &       &  &  &      &  &  &      &  \\
         &      &         &-1&1 &\dots &1 &       &  &  &      &  &  &      &  \\
         &      &         &  &  &      &  &\ddots &  &  &      &  &  &      &  \\
         &      &         &  &  &      &  &       &-1&1 &\dots &1 &0 &      &  \\
         &      &         &  &  &      &  &       &  &  &      &-1&1 &\dots &1 \\ 
\end{array}\right),$$
where there are exactly $a_i$ non-zero entries in row $i = 2, \dots, \ell$, all $\pm 1$ as above, and $\{\sigma_i\}^r_{i=1}$ forms a changemaker set. If, on the other hand, $q = 1$, then the last row of $A$ can be chosen to have the form
$$\left(\begin{array}{ccccc} \sigma^\prime_r &\dots &\sigma^\prime_1 &\sigma^\prime_0\end{array}\right),$$
where $\{\sigma^\prime_i\}^r_{i=0}$ forms a changemaker set.
\end{thm}

\begin{figure}
\centering
\includegraphics[width=0.3\textwidth]{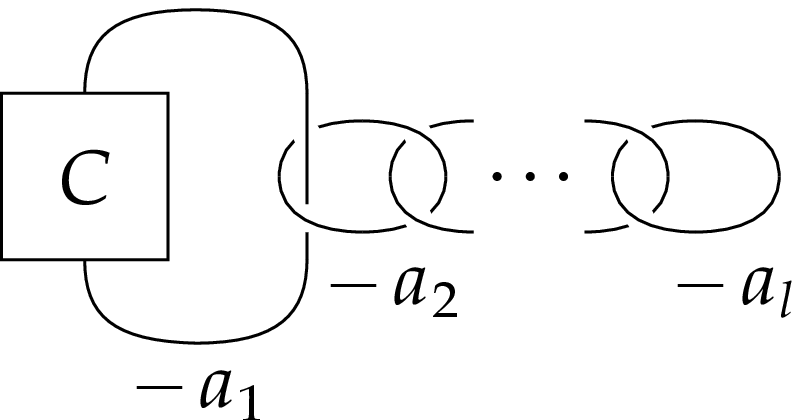}
\caption{A Kirby diagram for $Y = S^3_{-p/q}(C)$, where $p/q = [a_1,\dots,a_\ell]$.}
\label{trace}
\end{figure}

\subsection*{Acknowledgements} The author would like to thank his supervisor, Dr. Dorothy Buck, as well as Drs. Joshua Greene and Brendan Owens for their considerable support (in particular, drawing his attention to \cite{NiWu}). He is supported by the Rector's Award, SOF stipend, and Roth Fellowship at Imperial College London.
\section{An application of knot Floer homology}

If $C \subset S^3$ is a knot, then recall that the associated knot Floer chain complex $CFK:=CFK^\infty(S^3,C)$ is the $\mathbb{Z}$-module generated by a set $X$ together with a filtration $\mathcal{I} : X \rightarrow \mathbb{Z}\oplus\mathbb{Z}$ satisfying the properties
\begin{enumerate}
\item $\mathcal{I}(U\cdot x) = (i-1,j-1)$ if $\mathcal{I}(x) = (i,j)$; and
\item $\mathcal{I}(y) \leq \mathcal{I}(x)$ for all $y$ with non-zero coefficient in $\partial x$.
\end{enumerate}
Let $S$ be a subset of $\mathbb{Z}\oplus\mathbb{Z}$ such that $(i,j)\in S$ implies $(i+1,j),(i,j+1) \in S$, and define $CFK\{S\}$ to be the quotient of the knot Floer complex by the submodule generated by those $x \in X$ with $\mathcal{I}(x) \in S$. In this notation, we let
$$A^+_k := CFK\{i \geq 0 \text{ or } j \geq k\} \qquad \text{ and } \qquad B^+_k := CFK\{i \geq 0\},$$
where $k \in \mathbb{Z}$. As per \cite{OSRat}, these complexes come equipped with canonical $U$-equivariant chain maps
$$v^+_k,h^+_k : A^+_k \longrightarrow B^+_k$$
such that $v^+_k$ is projection onto $CFK\{i\geq 0\}$ and $h^+_k$ is a composition of projection onto $CFK\{j\geq k\}$, identification with $CFK\{j \geq 0\}$, and chain homotopy equivalence with $CFK\{i \geq 0\}$. At sufficiently high gradings, these maps are isomorphisms and hence behave as multiplication by $U^{V_k}$ and $U^{H_k}$ respectively where $V_k,H_k \geq 0$ are integers. The following lemma is taken from \cite{NiWu}.

\begin{lem}
\label{rel}
The $V_i$ and $H_i$ satisfy the following properties:
\begin{enumerate}
\item $V_0 = H_0$, and all $V_i,H_i \geq 0$;
\item The $V_i$ are a non-increasing sequence, while the $H_i$ are a non-decreasing sequence.
\end{enumerate}
\end{lem}

Using the labelling of $\Spinc$-structures given in Section 7 of \cite{OSRat}, Ni and Wu proved the following proposition about $D^{p/q}_C(\mathfrak{t})$. It can be found in \cite{NiWu} as Proposition 2.11, though as stated here we have applied it to $\overline{C}$.

\begin{prop}[Ni-Wu]
\label{NW}
Let $C$ be any knot in $S^3$, and let $p,q > 0$ be coprime. Then
$$D^{p/q}_C (\mathfrak{t}_i) = 2\max \left\{V_{\floor{\tfrac{i}{q}}},H_{\floor{\tfrac{i-p}{q}}}\right\}.$$
\end{prop}

As it stands, the labelling $\mathfrak{t}_i$ used above is a difficult one to manipulate with general rational surgeries, but simplifies considerably in the case of an integral $n$-surgery. In this instance, the $\Spinc$-structure $\mathfrak{t}_i$ is the one that admits an extension $\mathfrak{s}$ over the cobordism $S^3 \rightarrow S^3_{n}(C)$ which satisfies
$$\left<c_1(\mathfrak{s}),[F]\right> \equiv n + 2i \mod 2n,$$
where $F$ is a Seifert surface glued to the core of the attached handle. This fact will be useful to us later in our proofs.

We can now give a proof of the following result.

\begin{lem}
\label{eq}
Let $C$ be a knot in $S^3$. Then
\begin{align*}
\sum_{\mathfrak{t}\in\Spinc(S^3_{-p/q}(C))} D^{p/q}_C(\mathfrak{t}) = q\cdot&\sum_{\mathfrak{t}\in \Spinc(S^3_{-n}(C))} D^n_C (\mathfrak{t}) - r\cdot \min_{\mathfrak{t} \in \Spinc(S^3_{-n}(C))} \left\{ D^n_C(\mathfrak{t}) \right\}.
\end{align*}
\end{lem}
\begin{proof}
We consider the integral surgery first. By a direct application of Proposition \ref{NW} we obtain
\begin{equation}
\label{integral}
\sum_{\mathfrak{t}\in\Spinc(S^3_{-n}(C))} D^n_C(\mathfrak{t}) = 2\sum_{i=0}^{n-1} \max\left\{V_i,H_{i-n}\right\}.
\end{equation}
Our goal is to compare this with the rational surgery.

Labelling the $\Spinc$-structures on the rational surgery as $\mathfrak{t}_{iq+j}$, we have the following bounds:
\begin{enumerate}
\item $j$ ranges from $0$ to $q-1$;
\item $i$ ranges from $0$ to $n-1$ if $j < q-r$, or from $0$ to $n-2$ if $j \geq q-r$.
\end{enumerate}
Rephrasing the second of these, $i$ ranges from $0$ to $n-1-\delta(j)$, where $\delta(j) := \floor{\tfrac{j+r}{q}}$. Consequently, using Proposition \ref{NW},
\begin{equation}
\label{spincs}
\sum_{j=0}^{q-1}\sum_{i=0}^{n-1-\delta(j)} D^{p/q}_C(\mathfrak{t}_{iq+j}) = 2\sum_{j=0}^{q-1}\sum_{i=0}^{n-1-\delta(j)}\max\{V_i,H_{i-n+\delta(j)}\}.
\end{equation}
We fix $j$ and observe that
\begin{equation}
\label{combine}
\sum_{i=0}^{n-1-\delta(j)}\max\{V_i,H_{i-n+ \delta(j)}\} = \begin{cases}
\sum_{i=0}^{n-1}\max\{V_i,H_{i-n}\} & \text{if }\delta(j)=0\\
\sum_{i=0}^{n-2}\max\{V_i,H_{i-n+1}\} & \text{if }\delta(j)=1\end{cases}.
\end{equation}
Clearly, if $\delta(j) = 0$, then the RHS is the same as the RHS of \eqref{integral}. This happens for the first $q-r$ values of $j$, meaning that the situations of interest are the $r$ larger cases when $\delta(j) = 1$. In effect, in order to obtain our result we need to establish that
$$\sum_{i=0}^{n-2}\max\{V_i,H_{i-n+1}\} = \sum_{i=0}^{n-1}\max\{V_i,H_{i-n}\} - m,$$
where $2m$ is the value of the minimum deficiency.

Suppose that $i^\prime$ is chosen to be the largest integer such that the integral deficiency in $\Spinc$-structure $\mathfrak{t}_{i^\prime}$ is minimal. That is, that $\max \{V_{i^\prime},H_{i^\prime-n}\}$ is minimal. Then there are two possibilities.

\begin{enumerate}
\item Suppose that $V_{i^\prime} \geq H_{i^\prime-n}$. As $V_*$ is non-increasing and $H_*$ is non-decreasing, it follows that $V_i \geq H_{i-n}$ for all $i \leq i^\prime$, and hence that $V_i \geq V_{i+1} \geq H_{i-n+1}$ for all $i < i^\prime$.

\qquad Going in the other direction, suppose that $V_{i^\prime + 1} > H_{i^\prime-n+1}$. Then our choice of $i^\prime$ implies that $V_{i^\prime + 1} = \max\{V_{i^\prime+1},H_{i^\prime -n + 1}\} > \max\{V_{i^\prime},H_{i^\prime - n}\} =  V_{i^\prime}$, a contradiction to the non-increasing behaviour of $V_*$. Thus $H_{i^\prime - n + 1} \geq V_{i^\prime + 1}$, and $H_{i-n+1} \geq H_{i-n} \geq V_i$ for all $i > i^\prime$.

\qquad In the case $i = i^\prime$, observe that
$$V_{i^\prime} = \max\{V_{i^\prime},H_{i^\prime-n}\} < \max\{V_{i^\prime+1},H_{i^\prime -n + 1}\} = H_{i^\prime - n +1}.$$
Putting this together with the conclusions of the previous two paragraphs, we deduce that
\begin{align*}
\sum_{i=0}^{n-2}\max\{V_i,H_{i-n+1}\} &= \sum_{i=0}^{i^\prime - 1}V_i + \sum_{i=i^\prime}^{n-2} H_{i-n+1} \\
&= \sum_{i=0}^{n-1}\max\{V_i,H_{i-n}\} - V_{i^\prime}.
\end{align*}
Observe that $V_{i^\prime} = m$.
\item Suppose instead that $H_{i^\prime-n}\geq V_{i^\prime}$. This case is similar in nature, though a little more complicated: we define $j^\prime$ to be the smallest integer such that $\max\{V_k,H_{k - n}\}$ is minimal for $j^\prime \leq k \leq i^\prime$, and end with the conclusion
\begin{align*}
\sum_{i=0}^{n-2}\max\{V_i,H_{i-n+1}\} &= \sum_{i=0}^{j^\prime-1} V_i + \sum_{i=j^\prime}^{i^\prime-1} H_{i^\prime - n} + \sum_{i=i^\prime}^{n-2} H_{i-n+1}\\
&=\sum_{i=0}^{n-1}\max\{V_i,H_{i-n}\} - H_{i^\prime-n}
\end{align*}
and the observation that $H_{i^\prime - n} = m$, by definition of $i^\prime$.
\end{enumerate}
To complete the proof, one puts the above information into \eqref{spincs} via \eqref{combine} and compares with \eqref{integral}.
\end{proof}

If one reads this argument carefully, one will find that it can be modified slightly to give a proof of Theorem \ref{main}. However, since this modified argument provides no insight as to the nature of $\mathfrak{r}$ without a deeper knowledge of the labelling $\mathfrak{t}_i$, it is of limited use to us.

In light of the above lemma, a natural question at this point is: Which elements of $\Spinc(S^3_{-n}(C))$ minimise $D^n$? The answer is below.

\begin{lem}
\label{minimising}
According to the parity of $n$,
\begin{enumerate}
\item If $n$ is even, then $\mathfrak{t}_{\tfrac{n}{2}}$ realises the minimal deficiency; and
\item If instead $n$ is odd, then $\mathfrak{t}_{\tfrac{n\pm1}{2}}$ do the same.
\end{enumerate}
\end{lem}
\begin{proof}
Recall that $\mathfrak{t}_i$ evaluates, modulo $2n$ to $n + 2i$. Hence $\mathfrak{t}_i$ and $\mathfrak{t}_{n-i}$ are conjugates, and so, by the conjugation symmetry of correction terms, if $i \neq 0$ and $\mathfrak{t}_i$ realises the minimum, so does $\mathfrak{t}_{n-i}$. Assuming that $i \leq n-i$, we claim that the same is true for all $\mathfrak{t}_j$ with $i \leq j \leq n-i$. Indeed, let the minimum deficiency be $2m$, so that $\max\{V_i,H_{i-n}\} = \max\{V_{n-i},H_{-i}\} = m$. We know that $m \geq V_i \geq V_j$ and $m \geq H_{-i} \geq H_{j-n}$, so $\max \{V_j,H_{j-n}\} \leq m$, and as $m$ is minimal it follows that we have equality. Consequently, $\mathfrak{t}_j$ also realises the minimum.

Thus, if $\mathfrak{t}_i$ realises the minimum for $i \neq 0$, so do the $\mathfrak{t}_j$ for the centralmost values of $j$, namely $\tfrac{n}{2}$ or $\tfrac{n\pm1}{2}$, depending on parity. The other possibility, of course, is that $\mathfrak{t}_0$ realises the minimum. In this case, observe that $D^n_C(\mathfrak{t}_0) = 2\max\{V_0,H_{-n}\}$. By Lemma \ref{rel}, $V_0 = H_0 \geq H_{-n}$, and we see that the deficiency is $2V_0$. Since $V_i \leq V_0$ and $H_{i-n} \leq H_0 = V_0$,
$$D^n_C(\mathfrak{t}_i) = 2\max \{V_i,H_{i-n}\} \leq 2V_0 = D^n_C(\mathfrak{t}_0),$$
and $\mathfrak{t}_0$ is in fact the $\Spinc$-structure with the maximal deficiency.
\end{proof}
\section{Preliminaries to the proofs}

Our goal now is to exhibit a function $\mathfrak{r} : \Spinc(S^3_{-p/q}(C)) \rightarrow \Spinc(S^3_{-n}(C))$ that satisfies Theorem \ref{main}. This will we require some enumeration of the $\Spinc$-structures on $S^3_{-p/q}(C)$, but because of the bijection
$$\Spinc(S^3_{-p/q}(C) \longleftrightarrow \Spinc(S^3_{-p/q}(U)$$
described in Section \ref{introduction}, this enumeration need only consider the case $C = U$.

\subsection{A Plumbing Diagram for $S^3_{-p/q}(U)$}

Consider the linear graph $G$ in Figure \ref{linear}, where $p/q = [a_1,\dots,a_\ell]$ is written in Hirzebruch-Jung continued fraction notation and $a_i \geq 2$ for $i \geq 2$ (recall that such an expansion always exists). Then $G$ determines a sharp, simply connected, negative-definite smooth 4-manifold $W^\prime$ with free $H_2(W^\prime)$ by plumbing (see \cite{OSPlumbed}). Moreover, $\partial W^\prime = S^3_{-p/q}(U)$.

\begin{figure}
\centering
\includegraphics[width=0.3\textwidth]{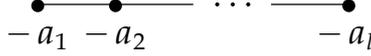}
\caption{A weighted graph $G$ with weighted adjacency matrix $Q$. Vertices are labelled $v_1$ to $v_\ell$ from left to right, and have weights $w(v)$ where $v \in V(G)$, the vertex set. The integers $a_i$ are taken from the Hirzebruch-Jung continued fraction expansion $p/q = [a_1,\dots,a_\ell]$ where $a_i \geq 2$ for $i \geq 2$. Notice that $G$ determines a 4-manifold $W^\prime$ by plumbing, and that $\partial W^\prime = S^3_{-p/q}(U)$.}
\label{linear}
\end{figure}

Following \cite{OSPlumbed}, it is possible to use this diagram to enumerate $\Spinc(S^3_{-p/q}(U))$ and compute the corresponding correction terms. Indeed, we observe that $H_1(W^\prime)$ is generated by $[v]$, where $v$ is a vertex of $G$, and $H^2(W^\prime,\partial W^\prime)$ by $\PD[v]$. Pushing this through the short exact sequence
$$\begin{CD} 0 @>>> H^2(W^\prime,\partial W^\prime) @>{Q}>> H^2(W^\prime) @>{\alpha}>> H^2(\partial W^\prime) @>>> 0,\end{CD}$$
we see that $\ker \alpha$ is generated by the images of the $\PD[v]$ in $H^2(W^\prime)$ (i.e. by the rows of $Q$). We will think of $H^2(W^\prime)$ as $\Hom(H_2(W^\prime),\mathbb{Z})$, since $W^\prime$ is simply connected.

Now, consider all the characteristic covectors $\Char(G) \subset H^2(W^\prime)$. That is, those $K$ which satisfy
$$\left< K, [v] \right> \equiv w(v) \mod 2 \text{ for all } v \in V(G).$$
Since $c_1 : \Spinc(W^\prime) \rightarrow H^2(W^\prime)$ is injective with image $\Char(G)$, we shall think of these covectors as the $\Spinc$-structures on $W^\prime$. If $K \in \Char(G)$, so $K = c_1(\mathfrak{s})$ for some $\mathfrak{s} \in \Spinc(W^\prime)$, then we shall write $[K]$ for the $\Spinc$-structure $\mathfrak{s}\vert_{\partial W^\prime}$, and thus partition $\Char(G)$ into equivalence classes using the relation $K \sim K^\prime$ if $[K] \sim [K^\prime]$ (where $K,K^\prime \in \Char(G)$). Because $W^\prime$ is sharp (see \cite{OSPlumbed}), every $\mathfrak{t} \in \Spinc(\partial W^\prime)$ lifts to some $\mathfrak{s} \in \Spinc(W^\prime)$, and hence we can think of any complete set of representatives of equivalence classes of $\sim$ as being the $\Spinc$-structures on $\partial W^\prime$.

As we would ideally like to list such representatives, it is fortunate then that the results of \cite{OSPlumbed} tell us exactly how to do this. In that paper, Ozsv\'ath and Szab\'o prove that $\ker U \subset HF^+(\partial W^\prime)$ is given by some subset of those $K$ satisfying
\begin{equation}
w(v) + 2 \leq \left<K,v\right> \leq -w(v) \text{ for all } v \in V(G).
\label{plumb}
\end{equation}
To determine which such $K$, we start with some $K$ satisfying \eqref{plumb} and let $K_0 := K$. If $\left<K_i,v\right> = -w(v)$ for some $v$, we set $K_{i+1} := K_i + 2PD[v]$. Notice that $K^\prime \sim K$. This operation is called \emph{pushing down} (the co-ordinate of) $K_i$ at $v$. Continuing like this, we conclude either with some $L := K_m$ such that
$$w(v) \leq \left<L,v\right> \leq -w(v) -2 \text{ for all } v \in V(G),$$
or else with an $L$ such that there exists a $v$ satisfying $\left<L, v\right> \geq -w(v)$. If we conclude in the first way, we say that $K$ initiates a \emph{maximising} path. If we conclude in the second, we say that $K$ initiates a \emph{non-maximising} path. The relevant result from \cite{OSPlumbed} is that $\ker U$ is given by those $K$ that satisfy \eqref{plumb} and initiate maximising paths. As a corollary from the same paper, their correction terms are computed using the formula
\begin{equation}
\label{dformula}
d(\partial W^\prime, \mathfrak{t}) = \max_{K : [K] = \mathfrak{t}}\frac{KQ^{-1}K^t + \abs{G}}{4}.
\end{equation}
We remind the reader that as $\partial W^\prime$ an $L$-space, $\ker U$ gives us the complete collection of $\Spinc$-structures on $\partial W^\prime$ without repetition.

\subsection{Correction Term Calculus}

At this point, it is natural to ask what the $\Spinc$-structures on $S^3_{-p/q}(U)$ look like after applying the above algorithm. To answer this question, we require some definitions.

\begin{defn}
Suppose that $Y$ is a closed 3-manifold contained in $X$, a smooth 4-manifold. Then given some $\mathfrak{s} \in \Spinc(X)$, we say that $c_1(\mathfrak{s})$ is a \emph{maximiser} for $\mathfrak{t} = \mathfrak{s}\vert_Y$ if $c_1(\mathfrak{s})^2$ is maximal among $c_1(\mathfrak{s}^\prime)^2$, where $\mathfrak{s}^\prime \in \Spinc(X)$ satisfies $\mathfrak{s}^\prime\vert_Y = \mathfrak{t}$.
\end{defn}

\begin{defn}
Let $K$ be a characteristic covector for the linear graph $G$ in Figure \ref{linear} which satisfies
$$w(v) \leq \left<K,v\right> \leq -w(v) \text{ for all } v \in V(G).$$
Then we say that a vertex $v_i$ is a \emph{peak} for $K$ if $\left<K,[v_i]\right> = a_i$, and call $K + 2\PD[v_i]$ the \emph{push-down} of $K$ at $v_i$ (we also call any covector obtained by a sequence of such moves a \emph{push-down} of $K$). We say that $K$ contains no \emph{full tanks} if there do not exist $i < j$ such that $v_i$ and $v_j$ are peaks and $\left<K,[v_k]\right>= a_k - 2$ for all $i < k < j$. We say that $K$ is \emph{left-full} if there exists a peak $v_i$ such that $\left<K,[v_k]\right> = a_k-2$ for all $k < i$. To make notation simpler we will write $b_k := 2 - a_k$.
\end{defn}

\begin{lem}
The $\Spinc$-structures on $S^3_{-p/q}(U)$ are represented by those characteristic covectors $K$ satisfying \eqref{plumb} that contain no full tanks. We call the collection of such characteristic covectors $\mathcal{K}$.
\end{lem}
\begin{proof}
We first prove that if $K$ has a full tank then it initiates a non-maximising path. Indeed, observe the following path (presenting only the relevant section of $K$):
\begin{align*}
(a_i,-b_{i+1},-b_{i+2}, \dots, -b_{j-1},a_j) &\longrightarrow (-a_i,a_{i+1},-b_{i+2},\dots, -b_{j-1},a_j)\\
&\longrightarrow (b_i,-a_{i+1},a_{i+2},\dots, -b_{j-1},a_j)\\
&\longrightarrow (b_i,b_{i+1},-a_{i+2},\dots, -b_{j-1},a_j)\\
&\longrightarrow \dots \\
&\longrightarrow (b_i,b_{i+1},b_{i+2},\dots, a_{j-1},a_j)\\
&\longrightarrow (b_i,b_{i+1},b_{i+2},\dots, -a_{j-1},a_j+2).
\end{align*}
Here $\left<L,[v_j]\right> > -w(v_j)$, so the initiated path is non-maximising.

What remains to be shown is that if $K$ does not have a full tank, then it initiates a maximising path. We do this by inducting on the number peaks in $K$ and its push-downs. If there are none, we have a (trivial) maximising path. Thus, we presume there is at least one peak at $v_i$. Now, push down at $v_i$. Depending on whether $\left<K,[v_{i\pm1}]\right>=a_{i\pm1}-2$, there are three possibilities for the new $K^\prime = K + 2\PD[v_i]$:
\begin{enumerate}
\item $v_{i-1}$ and $v_{i+1}$ are not peaks of $K^\prime$. Then $K^\prime$ has one peak fewer than $K$ and also contains no full tanks as $\left<K^\prime,[v_i]\right> = -a_i \neq a_i - 2$ since $a_i \geq 2$ when $i \geq 2$. Hence, we apply the inductive hypothesis.
\item $v_{i-1}$ is not a peak, $v_{i+1}$ is (or the reverse situation). In this case, push down at $v_{i+1}$, and continue pushing down at any further peaks this generates, necessarily heading to the right. As $K$ had no full tanks, this process must stop without initiating a non-maximising path. As in the previous case, the resulting covector has one peak fewer than $K$ and no full tanks, so apply the induction hypothesis.
\item $v_{i-1}$ and $v_{i+1}$ are peaks. This situation is the same as the one above, pushing down in both directions unilaterally until the process halts. If $a_i = 2$, we will have to repeat this whole procedure multiple times, but eventually it will halt.
\end{enumerate}
In all situations we have a maximising path. This completes our proof.
\end{proof}

Since $S^3_{-p/q}(U)$ is an $L$-space, we have now isolated a collection $\mathcal{K} \subset \Char(G)$ in bijection with $\Spinc (S^3_{-p/q}(U))$. Hence, we have the following proposition (the last part of which is an application of \eqref{dformula}).

\begin{prop}
\label{sharpprop}
Given $\mathfrak{t} \in \Spinc(S^3_{-p/q}(U))$, there is a unique $K \in \mathcal{K}$ such that $[K] = \mathfrak{t}$ and $K$ is a maximiser for $[K]$. Moreover,
$$d(S^3_{-p/q}(U),\mathfrak{t}) = \frac{KQ^{-1}K^t + b_2(W^\prime)}{4}.$$
\end{prop}

\subsection{Comparing the Rational and Integral Surgeries}

We are now ready to make comparisons between the $-p/q$ and $-n$ surgeries on $C$ and $U$. In doing so, it is very important to keep track of which coefficients and which knots we are considering. Thus, we observe the following:
\begin{enumerate}
\item Let $W(C) : S^3 \rightarrow S^3_{-p/q}(C)$ be the cobordism determined by the diagram in Figure \ref{trace}, and let $W := W(C) \cup_{S^3} D^4$ and $W^\prime := W(U) \cup_{S^3} D^4$ (i.e. these manifolds are the traces of the surgeries in Figure \ref{trace}). Note that this $W^\prime$ and the $W^\prime$ of the previous section are identical;
\item The intersection form of the cobordism $W(C)$ is independent of $C$. Ergo, $W$ and $W^\prime$ have the same intersection form, represented in some bases by the adjacency matrix $Q$ of the graph $G$ in Figure \ref{linear};
\item Courtesy of this fact, a $K \in \Char(G)$ is a maximiser for $[K] \in \Spinc(S^3_{-p/q}(U))$ if and only if $K$ is also a maximiser for the corresponding $\Spinc$-structure on $S^3_{-p/q}(C)$, which we shall also denote by $[K]$. The crucial difference is that $K$ might not compute the correction term for $S^3_{-p/q}(C)$. Henceforth, we shall think of $\Char(G)$ as the $\Spinc$-structures on either 4-manifold $W$ or $W^\prime$, and any complete collection $\mathcal{F}$ of representatives of equivalence classes of $\sim$ as the $\Spinc$-structures on either 3-manifold $S^3_{-p/q}(C)$ or $S^3_{-p/q}(U)$;
\item $W(C)$ splits naturally into two cobordisms,
$$\begin{CD} S^3 @>{W_1(C)}>> S^3_{-n}(C) @>{W_2(C)}>> S^3_{-p/q}(C). \end{CD}$$
All three cobordisms are negative definite and have intersection forms independent of $C$.
\end{enumerate}

Now consider $W_2(C)$. Given a maximiser $K$ which determines $\mathfrak{t} \in \Spinc(S^3_{-p/q}(C))$, this $K$ also determines a $\mathfrak{t}^\prime \in \Spinc (S^3_{-n}(C))$ by considering the value of $k := \left<K,[v_1]\right>$ modulo $2n$. Comparing this with the labelling used in Proposition \ref{NW}, we see that $\mathfrak{t}^\prime$ corresponds with $\mathfrak{t}_{\tfrac{k + n}{2}} \in \Spinc(S^3_{-n}(C))$, unless $k = n$, in which case $\mathfrak{t}^\prime$ corresponds with $\mathfrak{t}_0 \in \Spinc(S^3_{-n}(C))$.

\begin{lem}
\label{cover}
Let $K \in \mathcal{K}$. Then,
\begin{enumerate}
\item If $\left<K,[v_1]\right> \neq n$, then there is a $q$-to-one map $\mathcal{K} \rightarrow \Char(v_1)$ given by restriction to the first co-ordinate;
\item If $\left<K,[v_1]\right> = n$, then there are $q-r$ elements of $\mathcal{K}$ which map to $(n) \in \Char(v_1)$, also by restriction.
\end{enumerate}
\end{lem}
\begin{proof}
Consider the case when $\left<K,[v_1]\right> \neq n$. Then the number of covectors $K$ that restrict to $\left<K,[v_1]\right>$ is equal to the number of characteristic covectors on the linear graph $G - v_1$ which satisfy \eqref{plumb}, initiate maximising paths, and have no full tanks. This is just the number of $\Spinc$-structures on the lens space given by $[a_2,\dots,a_\ell]$-surgery on the unknot. As $[a_2,\dots,a_\ell] = q/r$, we are done.

For the second claim, observe that there are $n-1$ values of $\left<K,[v_1]\right>$ covered by the above case, together accounting for $(n-1)q$ of the elements of $\mathcal{K}$. Therefore, the remaining $q-r$ must restrict to $(n) \in \Char(v_1)$.
\end{proof}

\subsection{Adjusting the Maximisers}

It is proved in \cite{OSAbsolute} that if $\overline{W}$ is a negative-definite cobordism from $Y_1$ to $Y_2$, both rational homology spheres, then
\begin{equation}
\left\{\frac{c_1(\overline{\mathfrak{s}})^2 + b_2(\overline{W})}{4}\right\} \leq d(Y_2,\overline{\mathfrak{s}}\vert_{Y_2}) - d(Y_1,\overline{\mathfrak{s}}\vert_{Y_1})
\label{cobord}
\end{equation}
for any $\overline{\mathfrak{s}}\in\Spinc(\overline{W})$. Applying this to $\overline{W} = W_2(U)$, we observe that if $\overline{\mathfrak{s}}$ is the restriction of some $\mathfrak{s} \in \Spinc(W^\prime)$ satisfying $c_1(\mathfrak{s}) \in \mathcal{K}$, then we have equality in \eqref{cobord} (c.f. the proof of Lemma 4.3 in \cite{OSUnknot} and Proposition \ref{sharpprop}). This allows us to compute the term in curly braces and, since the intersection form of $W_2(C)$ is independent of $C$, substitute it into \eqref{cobord} for $\overline{W} = W_2(C)$. Consequently, if $\mathfrak{t}^\prime$ and $\mathfrak{t}$ are cobordant via an element of $\mathcal{K}$,
\begin{equation}
D^n_C(\mathfrak{t}^\prime) \leq D^{p/q}_C(\mathfrak{t}).
\label{spincdif}
\end{equation}

A very similar argument, applying \eqref{cobord} to $W_1(C)$, tells us that
\begin{equation}
0 \leq D^n_C(\mathfrak{t}^\prime).
\label{spincbound}
\end{equation}

This argument on $\mathcal{K}$ is just a special case of the following lemma.

\begin{lem}
\label{inequalities}
Suppose that $\mathcal{F} \subset \Char(G)$ is a complete set of representatives of equivalence classes of $\sim$, and that every $K_i \in \mathcal{F}$ is a maximiser for $[K_i]$. Then on defining $\mathfrak{s}_i$ by $c_1(\mathfrak{s}_i) = K_i$ and setting $w_i = \mathfrak{s}_i\vert_{S^3_{-p/q}(C)}$ and $v_i = \mathfrak{s}_i\vert_{S^3_{-n}(C)}$, it follows that
$$0 \leq D^n_C(v_i) \leq D^{p/q}_C(w_i).$$
\end{lem}
\begin{proof}
This is a combination of \eqref{spincbound} on the left and \eqref{spincdif} on the right.
\end{proof}

As it turns out, $\mathcal{K}$ is not the optimal choice of representatives for our purposes, since it does not yield a function $\mathfrak{r}$ satisfying Theorem \ref{main}. We therefore ask: If $K$ is a maximiser, are there any other $K^\prime \sim K$ that are maximisers? The answer is yes.

\begin{lem}
Let $\left<K,[v_i]\right> = a_i$, where $K \in \Char(G)$ (not necessarily in $\mathcal{K}$). Then $K^\prime: = K + 2\PD[v_i]$ satisfies $(K^\prime)^2 = K^2$.
\end{lem}
\begin{proof}
Recall that $\PD[v_i]$, viewed as an element of $H^2(W)$, is the $i$-th row of $Q$. Hence, $\PD[v_i]Q^{-1} = e_i$, the $i$-th standard basis vector. Thus,
\begin{align*}
(K+2\PD[v_i])^2 &= (K+2\PD[v_i])Q^{-1}(K+2\PD[v_i])^t \\
&= KQ^{-1}K^t + 4\PD[v_i]Q^{-1}K^t + 4\PD[v_i]Q^{-1}\PD[v_i]^t\\
&= KQ^{-1}K + 4e_iK^t + 4e_i\PD[v_i]^t\\
&= KQ^{-1}K + 4\left<K,[v_i]\right> - 4a_i,
\end{align*}
and as $\left<K,[v_i]\right> = a_i$, we are done.
\end{proof}

\begin{cor}
\label{sharp}
If $K^\prime$ is a push-down of $K \in \mathcal{K}$, then $K^\prime$ is a maximiser of $[K]$.
\end{cor}

\begin{cor}
Let $\mathcal{M}$ be the set of all maximisers in $\Char(G)$. Then if $K \in \mathcal{M}$, so are all its push-downs.
\end{cor}

Finally, we are able to prove the following critical lemma of this section.

\begin{lem}
\label{totalineq}
Let $C$ be a knot in $S^3$, then
\begin{align*}
\sum_{\mathfrak{t} \in \Spinc(S^3_{-p/q}(C))} D^{p/q}_C(\mathfrak{t}) \geq q \cdot &\sum_{\mathfrak{t} \in \Spinc(S^3_{-n}(C))} D^n_C(\mathfrak{t}) - r\cdot \min_{\mathfrak{t}\in\Spinc(S^3_{-n}(C))}\left\{ D^n_C(\mathfrak{t})\right\}.
\end{align*}
\end{lem}
\begin{proof}
Construct a family $\mathcal{K}^\prime$ of characteristic covectors for use in Lemma \ref{inequalities} as follows. If $K \in \mathcal{K}$ satisfies
\begin{enumerate}
\item $\left<K,[v_1]\right> = j$ for some $-1 \leq j < n$; and
\item $K\vert_{G-v_1}$ is left full,
\end{enumerate}
then let $K^\prime := K + 2\sum_{i=1}^k\PD[v_i]$ be a member of $\mathcal{K}^\prime$. Here $k \geq 2$ is the smallest integer such that $v_k$ is a peak for $K$ (guaranteed to exist by the second condition above). This clearly determines the same $\Spinc$-structure on the boundary manifolds, and is a maximiser by Corollary \ref{sharp}.

For all other $K \in \mathcal{K}$, let $K$ be a member of $\mathcal{K}^\prime$. The family $\mathcal{K}^\prime$ is now clearly a complete set of representatives for the equivalence classes of $\sim$, each element of which is a maximiser. We claim that the desired result is obtained by adding up all inequalities in Lemma \ref{inequalities}, using $\mathcal{F} = \mathcal{K}^\prime$.

To prove this claim, let us consider what the pushing down does. Our first piece of information is that $\left<K^\prime,[v_1]\right> = j+2$, so we are ``nudging $K$ up'' the values in the first co-ordinate. We claim that, for a given $j$, we have nudged up precisely $r$ different $K$. Indeed, recall from Lemma \ref{cover} that there are $q-r$ elements of $\mathcal{K}$ with $\left<K,[v_1]\right>=n$. Another way of computing this number is:
$$\#\mathcal{K}\vert_{G-v_1} - \#\left\{ \begin{array}{c}\text{left-full elements of}\\ \mathcal{K}\vert_{G-v_1}\end{array}\right\}.$$
Since the first term here is $q$ (as a scholium of Lemma \ref{cover}), the second term must be $r$, as required.

This calculation completed, we now observe that $\mathcal{K}^\prime$ has $q$ elements that restrict to $(j) \in \Char(v_1)$ for any $-n+2 \leq j \leq n$, except $j = -1$ if $n$ is odd or $j = 0$ if $n$ is even, when there are $q-r$ such elements. Our lemma follows by applying Lemma \ref{inequalities} and adding up all the inequalities. Notice that the exceptional $\Spinc$-structure is one with minimal deficiency (see Lemma \ref{minimising}).
\end{proof}
\section{Proofs of the theorems}

Now that all the machinery is in place, we can rapidly prove Theorem \ref{main}.

\begin{proof}[Proof of Theorem \ref{main}]
Take Lemma \ref{totalineq}, observing that we actually have equality by Lemma \ref{eq}. This implies that all the right hand inequalities in Lemma \ref{inequalities} were in fact equalities induced by the members of $\mathcal{K}^\prime$. The result follows.
\end{proof}

As remarked after the proof of Lemma \ref{eq}, we had actually already proved Theorem \ref{main} some time ago. However, this more recent proof has the advantage that it gives us insights the previous one did not: it allows us to see how $\mathfrak{r}$ behaves. Indeed, take any $\mathfrak{t} \in \Spinc(S^3_{-p/q}(C))$ and some maximiser $K$ for $\mathfrak{t}$. Then the value $\mathfrak{r}(\mathfrak{t})$ is determined by finding the $K^\prime \in \mathcal{K}^\prime$ such that $K \sim K^\prime$; $\mathfrak{r}(\mathfrak{t})$ is the $\Spinc$-structure on $S^3_{-n}(C)$ determined by the maximiser $(\left<K^\prime,[v_1]\right>)$.

\begin{cor}
\label{min}
With notation as above, $D^{p/q}(\mathfrak{t})$ is minimal if $\left<K^\prime,[v_1]\right> = 0, \pm 1$. If, additionally, $n$ is even and there are $q-r+1$ choices of $\mathfrak{t}$ such that $D^{p/q}(\mathfrak{t})$ is minimal, then this extends to $\left<K^\prime,[v_1]\right> = \pm 2$.

In either case, if $D^{p/q}(\mathfrak{t}) = 0$ for some $\mathfrak{t}$, then the minimal deficiency is zero.
\end{cor}
\begin{proof}
The first statement is an immediate consequence of Theorem \ref{main} and Lemma \ref{minimising}. The second arises because there are only $q-r$ such $\mathfrak{t}$ with first co-ordinate $0$, and because those $K^\prime$ with $\left<K^\prime,[v_1]\right> = \pm 2$ have the next smallest deficiencies (by an argument almost identical to that in Lemma \ref{minimising}). The final comment is a trivial by-product of Lemma \ref{inequalities} as the deficiencies are non-negative.
\end{proof}

As mentioned in Section \ref{introduction}, this knowledge of $\mathfrak{r}$ allows us to turn Theorem \ref{main} into an obstruction, given $Y$, to $Y = S^3_{-p/q}(C)$ (under certain extra circumstances). We have already stated this obstruction as Theorem \ref{changemaker}, but to prove it we must establish some algebraic preliminaries. For greater detail on these preliminaries, we refer the reader to Lemma 2.3 of \cite{GreeneGenus} and Section 3.2 of \cite{GreeneBraid}. We have summarised the key results below.

\begin{prop}
Let $Y$ be a 3-manifold obtained by integral surgery on an $\ell$-component link $L$ with negative-definite linking matrix $Q$ and trace $W$; let also $Y^\prime$ be a manifold obtained by integral surgery on a (possibly different) link $L^\prime$, also with linking matrix $Q$, but whose trace $W^\prime$ is sharp. Finally, suppose that $\partial X = -Y$, where $X$ is a sharp, negative-definite smooth 4-manifold, that $H_1(X) = 0$, and that $H_2(X)$ is free. Then since $c_1$ commutes with the restriction maps on $\Spinc(\cdot)$ and $H^2(\cdot)$ induced by inclusion of a 3- or 4-manifold into a 4-manifold, there is a bijection
\begin{align}
\notag
\left\{\mathfrak{s} \in \Spinc(X\cup_Y W)\left\vert \mathfrak{s}\vert_Y = \mathfrak{t}\right\}\right.  &\longrightarrow \left\{(\mathfrak{s}_X,\mathfrak{s}_W) \in \Spinc(X)\times \Spinc(W)\left\vert \begin{array}{c} \mathfrak{s}_X\vert_Y = \mathfrak{t} \\ \mathfrak{s}_W\vert_Y = \mathfrak{t}\end{array}\right\}\right. \\
\label{biject}
\mathfrak{s} &\longmapsto (\mathfrak{s}\vert_X,\mathfrak{s}\vert_W)
\end{align}
such that
$$c_1(\mathfrak{s})^2 = c_1(\mathfrak{s}\vert_X)^2 + c_1(\mathfrak{s}\vert_W)^2.$$
Moreover, given $\mathfrak{t} \in \Spinc(Y)$, there is some $\mathfrak{s} \in \Spinc(X \cup_Y W)$ such that
\begin{equation}
\label{sharpy}
\max_{\substack{\mathfrak{s}\in\Spinc(X\cup_Y W)\\\mathfrak{s}\vert_Y = \mathfrak{t}}}c_1(\mathfrak{s})^2 + b_2(X\cup_Y W) = 4d(Y^\prime,\mathfrak{t}) - 4d(Y,\mathfrak{t}).
\end{equation}
\end{prop}

Note that in \cite{GreeneBraid}, the additional assumption was made that $\det(Q)$ is odd. This, however, is not necessary: its only function was to ensure that $\Spinc(X) \rightarrow \Spinc(Y)$ and $\Spinc(W) \rightarrow \Spinc(Y)$ surject. This is assured by the fact that $H_1(X)$ and $H_1(W)$ are torsion-free (c.f. \cite{GreeneGenus}).

Taking this proposition as given, it follows that a maximiser $c_1(\mathfrak{s})$ decomposes into a pair of maximisers $(c_1(\mathfrak{s}\vert_X),c_1(\mathfrak{s}\vert_W))$. To see what this decomposition looks like, at least on $W$, we use the diagram
\begin{equation}
\label{algtop}
\begin{CD} H_2(X) \oplus H_2(W) @>>> H_2(X\cup_Y W) \simeq H^2(X \cup_Y W) @>>> H^2(X) \oplus H^2(W).\end{CD}
\end{equation}
If we employ a basis $\{u_1,\dots, u_\ell\}$ for $H_2(W)$ with images $\{\overline{u_1},\dots,\overline{u_\ell}\}$ in $H_2(X\cup_Y W)$, a class $\alpha \in H^2(W\cup_Y X)$ restricts to the class $(\left<\alpha,\overline{u_1}\right>,\dots,\left<\alpha,\overline{u_\ell}\right>)\in H^2(W)$ when written in the dual basis $\{u_1^*,\dots, u_\ell^*\}$. In particular this applies when $\alpha = c_1(\mathfrak{s})$: the restriction $c_1(\mathfrak{s}\vert_W)$ has the form $(\left<c_1(\mathfrak{s}),\overline{u_1}\right>,\dots,\left<c_1(\mathfrak{s}),\overline{u_\ell}\right>)$.

Now suppose that $K$ is a maximiser for $[K]$ and that $K = c_1(\mathfrak{s}_W)$ for some $\mathfrak{s}_W \in \Spinc(W)$. Suppose also that if we put $\mathfrak{t} = [K]$ in \eqref{sharpy}, the RHS vanishes. Then since $X$ is sharp, there is some $\mathfrak{s}^\prime \in \Spinc(X\cup_Y W)$ with $\alpha^\prime = c_1(\mathfrak{s}^\prime)$ such that
$$- b_2(X\cup_Y W) = (\alpha^\prime)^2 = c_1(\mathfrak{s}^\prime\vert_X)^2 + c_1(\mathfrak{s}^\prime\vert_W)^2.$$
As we know that $c_1(\mathfrak{s}^\prime\vert_W)$ is also a maximiser for $\mathfrak{t}$, it follows that its square is $K^2$. Thus, letting $\mathfrak{s} \in \Spinc(X\cup_Y W)$ correspond to $(\mathfrak{s}^\prime\vert_X, \mathfrak{s}_W)$ under \eqref{biject}, and putting $\alpha = c_1(\mathfrak{s})$, we have
$$(\alpha^\prime)^2 = c_1(\mathfrak{s}^\prime\vert_X)^2 + c_1(\mathfrak{s}^\prime\vert_W)^2 = c_1(\mathfrak{s}^\prime\vert_X)^2 + K^2 = c_1(\mathfrak{s}^\prime\vert_X)^2 + c_1(\mathfrak{s}_W)^2 = \alpha^2.$$
Hence $\alpha^2 + b_2(X\cup_Y W) = 0$.

Since $X\cup_Y W$ is a closed, simply connected, negative-definite smooth 4-manifold, it follows from Donaldson's diagonalisation theorem \cite{Donaldson} that
$$(H^2(X\cup_Y W),Q_{X\cup_Y W}) \simeq (\mathbb{Z}^{b_2(X)+b_2(W)},-\id)$$
as lattices. Thus, since $\alpha$ is a characteristic covector of $Q_{X\cup_Y W}$, it follows that $\alpha \equiv (1,1,\dots, 1) \mod 2$, whence all entries of $\alpha$ are $\pm 1$. Summarised, we have the following lemma.

\begin{lem}
\label{maximisers}
Let $K$ be a maximiser for $\mathfrak{t} = [K]$ such that
$$d(Y^\prime,\mathfrak{t}) - d(Y,\mathfrak{t})= 0.$$
Then there is some $\alpha \in \{\pm 1\}^{b_2(X)+\ell}$ such that $K = (\left<\alpha,\overline{u_1}\right>,\dots,\left<\alpha,\overline{u_\ell}\right>)$, written in the dual basis $\{u_1^*,\dots,u_\ell^*\}$.
\end{lem}

To apply this lemma, we let $Y^\prime = S^3_{-p/q}(U)$ (whose trace $W^\prime$ is sharp), and $Y = S^3_{-p/q}(C)$ (with trace $W$). Theorem \ref{changemaker} is now finally within reach.

\subsection{Proof of Theorem \ref{changemaker} when $p > q > 1$}
\label{mainproof}

Let our basis $\{u_1,\dots,u_\ell\}$ above be given by the vertices $[v_i]$ of our graph $G$, and consider the diagram below:
$$\begin{CD} \dots @>>> H^1(Y) @>>> H^2(X,\partial X) \oplus H^2(W,\partial W) @>>> H^2(X\cup_Y W) @>>> \dots \\
 & & @V{\PD}VV @V{\PD}VV @V{\PD}VV \\
\dots @>>> H_2(Y) @>>> H_2(X)\oplus H_2(W) @>{A}>> H_2(X \cup_Y W) @>>> \dots \end{CD}$$
Here, the two groups on the left must vanish because the vertical maps are isomorphisms and $H^1(Y) = 0$. Hence, the lattice underlying $Q_X\oplus Q_W$ embeds in the one underlying $Q_{X\cup_Y W}$, and on passing to the lower row, there must exist some matrix $A$ with integral entries such that $-AA^t = Q_X\oplus Q_W$. When expressed like this, the last $\ell$ rows of $A$ are the images of the $[v_i]$ in $H_2(X\cup_Y W)$, and we shall use Lemma \ref{maximisers} to prove our claims about their structure. It is helpful to keep in mind that the $(i,j)$-th entry of $Q_X\oplus Q_W$ is in fact the standard negative-definite inner product of the $i$-th and $j$-th rows of $A$. We label the last $\ell$ rows of $A$ by $x,y_2,\dots,y_\ell$.

Our first task is to establish the structure of $y_i$ for $i = 2, \dots, \ell$. This has three parts: first, we show that all the non-zero entries are unital; second, that non-adjacent rows have no non-zero entries in the same spots; and third, that adjacent rows share only one spot with non-zero entries and that these overlapping entries are opposite in sign.

To achieve the first of these objectives, consider $y_i = (y_{i,j})_j$ for $2 \leq i \leq \ell$. Recall that $b_j := 2-a_j$ and define $K \in \Char(G)$ by
$$K =\begin{cases}
(0,b_2,\dots,b_{i-1},a_i,b_{i+1},\dots,b_\ell) &\text{if }n \text{ is even} \\
(-1,b_2,\dots,b_{i-1},a_i,b_{i+1},\dots,b_\ell) &\text{if }n \text{ is odd}\end{cases}.$$
Since there are no full tanks in $K$ it is clear that $K \in \mathcal{K}$ and hence it is a maximiser. To determine the value of $D^{p/q}([K])$, we need to find the first co-ordinate of the corresponding $K^\prime \in \mathcal{K}^\prime$ such that $K^\prime \sim K$ (by Theorem \ref{main}). If there is some $a_j \neq 2$ for $2 \leq j < i$ then $K \in \mathcal{K}^\prime$. Otherwise, consider $K^\prime = K + 2\sum_{j=2}^i \PD[v_j] \in \mathcal{K}^\prime$, which has first co-ordinate $2$ or $1$ depending on parity. By our hypothesis on the deficiencies, it follows that $D^{p/q}([K])=0$ (via Corollary \ref{min}). Thus by Lemma \ref{maximisers}, there is an $\alpha \in \{\pm 1\}^{b_2(X)+\ell}$ such that $\left<\alpha,y_i\right> = a_i$. Rephrased,
$$- \sum_j \alpha_jy_{i,j} = a_i = \sum_j y_{i,j}^2,$$
where the right hand side comes from the fact that $y_i^2 = a_i$. Consequently,
$$\sum_j (y_{i,j}^2 + \alpha_j y_{i,j}) = 0,$$
and since $\alpha_j = \pm1$, each summand is non-negative. Therefore, each summand must vanish, which in turn requires $y_{i,j} \in \{-1,0,1\}$. It is clear that for exactly $a_i$ values of $j$, $y_{i,j} \neq 0$.

With this step done, we now need to establish how the rows line up with each other. Thus, consider $2 \leq i < j \leq \ell$ such that $j - i \geq 2$, set $m := a_i$, and permute the basis of $H_2(X\cup_Y W)$, changing signs as necessary, so that $y_i = (1, \dots, 1, 0, \dots, 0)$. As before, there must be an $\alpha = c_1(\mathfrak{s})$ such that $c_1(\mathfrak{s}\vert_W) = K$ where
$$K =\begin{cases}
(0,b_2,\dots,b_{i-1},a_i,-a_{i+1},b_{i+2},\dots,b_{j-2},b_{j-1},a_j,b_{j+1},\dots,b_\ell) &\text{if }n \text{ is even}\\(-1,b_2,\dots,b_{i-1},a_i,-a_{i+1},b_{i+2},\dots,b_{j-2},b_{j-1},a_j,b_{j+1},\dots,b_\ell) &\text{if }n \text{ is odd}\end{cases},$$
which is obtained by pushing down at $v_{j-1}$ in
$$(*,b_2,\dots,b_{i-1}, a_i-2,-b_{i+1},-b_{i+2}, \dots, -b_{j-2},a_{j-1},a_j-2,b_{j+1},\dots,b_\ell)\in \mathcal{K}$$
and repeating to the left. In either case, exactly as before we find that $D^{p/q}([K]) = 0$ by showing the first co-ordinate of the corresponding $K^\prime$ is one of $0,\pm 1,2$. Then there is an $\alpha$ such that $\left<\alpha,y_i\right> = a_i$ and $\left<\alpha,y_j\right> = a_j$. The first of these statements tells us that $\alpha_k = -1$ for all $k = 1, \dots, m$.

Now, let $I = \{ k \leq m \vert y_{j,k} \neq0\}$. We claim that $I = \emptyset$. Indeed, as $$-a_j = -\left<\alpha,y_j\right> = \sum_{k\in I} \alpha_k y_{j,k} + \sum_{k>m} \alpha_k y_{j,k},$$
each summand on the RHS must be $-1$ or $0$. We know that $\alpha_k = -1$ for $k \in I$, so $y_{j,k} = 1$ for $k \in I$. Yet $y_i \cdot y_j = 0$, so $\sum_{k \in I} 1 = 0$, and $I = \emptyset$.

We repeat a similar argument for $j = i+1$, though our goal is to show that there is a unique element $k\in I$ and that $y_{i+1,k} = -1$. For $2\leq i \leq \ell$ we take
$$K =
\begin{cases}
(0,b_2,\dots,b_{i-1},a_i,a_{i+1}-2,b_{i+2},\dots,b_\ell) &\text{if }n \text{ is even}\\
(-1,b_2,\dots,b_{i-1},a_i,a_{i+1}-2,b_{i+2},\dots,b_\ell) &\text{if }n \text{ is odd}\end{cases},$$
and note again that $D^{p/q}([K])=0$. Permuting and changing signs as necessary, we may assume that $y_i = (1,\dots,1,0,\dots,0)$ and define $I$ as before. This time, however,
$$-a_{i+1} +2 = -\left<\alpha,y_{i+1}\right> = \sum_{k\in I} \alpha_k y_{i+1,k} + \sum_{k>m} \alpha_k y_{i+1,k},$$
and exactly one summand on the RHS is $1$. If that summand is in the second sum then all summands in the first are negative, so $y_{i+1,k} = 1$ for all $k \in I$. But then $-1 = y_i \cdot y_{i+1} = \sum_{k\in I} 1$, a contradiction. Therefore $y_{i+1,k} = -1$ for precisely one $k \in I$, and by an argument similar to the one just made, this $k$ is the unique element of $I$.

At this point, up to permuting the basis of $H_2(X \cup_Y W)$ and changing signs as necessary, we have established the form of the last $\ell - 1$ rows of $A$. What remains is to establish $x$. With this in mind, our first goal is to prove that it has the shape $(*,\dots,*,1,0,\dots,0)$ as outlined in the statement of the theorem.

Fix $i \in \{3, \dots, \ell\}$, let $x_1,\dots,x_m$ be the entries of $x$ in the same spots as the non-zero entries of row $y_i$, and let $x_{m+1},\dots,x_{b_2(X)+\ell}$ be the rest. Note that $m = a_i$. Again, change signs as necessary so that $y_{i,k} \geq 0$ for all $k$. Then as $x \cdot y_i = 0$, it follows that
\begin{equation}
\label{xmzero}
x_1 + \dots + x_{m-1} + x_m = 0.
\end{equation}
Our goal is to show that $x_k = 0$ for all $k \leq m$.

Define a set
$$\mathcal{S} = \left\{K \in \mathcal{M} \left\vert \left<K,[v_i]\right>= m, D^{p/q}([K]) = 0\right\} \right. .$$
Then for any $K \in \mathcal{S}$, we find an $\alpha$ according to Lemma \ref{maximisers} such that $\left<\alpha,y_i\right> = a_i = m$. Hence, $\alpha_1 + \dots + \alpha_m = -m$, whence $\alpha_k = -1$ for $k \leq m$. Indeed, all $\alpha \in \{\pm 1\}^{b_2(X)+\ell}$ satisfying these equations determine some $K \in \mathcal{S}$. Let $j = \left<K,[v_1]\right> = \left<\alpha,x\right>$. Then
$$j = x_1 + \dots + x_m - \sum_{k>m}\alpha_k x_k = - \sum_{k>m}\alpha_k x_k,$$
where we used \eqref{xmzero} to obtain the last equality. Thus the maximum value for $j$ as we vary $K \in \mathcal{S}$ is
\begin{equation}
\label{maxjone}
\sum_{k>m} \abs{x_k}.
\end{equation}

Now suppose, without loss of generality, that $x_m \leq x_k$ for all $k < m$, and define
$$\mathcal{S}^\prime = \left\{K \in \mathcal{M} \left\vert \left<K,[v_i]\right> = m-2,D^{p/q}([K])=0\right\}\right. .$$
Then similarly there is some $\beta \in \{\pm1\}^{b_2(X)+\ell}$ such that $\left<\beta,y_i\right> = m-2$, whence $\beta_k=-1$ for all values of $k \leq m$ except one. As before, all such $\beta$ determine a $K \in \mathcal{S}^\prime$. Hence
\begin{equation}
\label{j}
j = -\sum_{k=1}^{m} \beta_k x_k - \sum_{k>m}\beta_kx_k
\end{equation}
attains its maximal value when $\beta_k = -1$ for all $k < m$ and $\beta_m = 1$ (by choice of $x_m$). This maximal value is
\begin{equation}
\label{maxjtwo}
\sum_{k < m} x_k - x_m +  \sum_{k>m}\abs{x_k}.
\end{equation}

We claim that the two maxima given by \eqref{maxjone} and \eqref{maxjtwo} are in fact identical. Indeed, let $j_{\max} \leq n$ be the maximal integer $j$ such that $D^n([j]) = 0$ (note that $j_{\max} \geq 1$ by assumption on the number of vanishing deficiencies). Then if $j_{\max}$ can be attained by elements of $\mathcal{S}$ and $\mathcal{S}^\prime$, the claim must be true (since larger values are ruled out by the deficiency condition). Observe that
$$K = (j_{\max},-a_2,b_3,\dots,b_{i-2},b_{i-1},m,b_{i+1},\dots,b_\ell)$$
satisfies $D^{p/q}([K]) = 0$, since $K \in \mathcal{K}^\prime$: it is obtained by pushing down $$(j_{\max}-2,-b_2,-b_3,\dots, -b_{i-2},a_{i-1},m-2,b_{i+1},\dots,b_\ell) \in \mathcal{K}$$
 at $v_{i-1}$ and to the left. Similarly,
$$K^\prime = (j_{\max},b_2,\dots,b_{i-1},m-2,b_{i+1},\dots,b_\ell) \in \mathcal{S}^\prime \cap \mathcal{K}^\prime.$$
Thus the maximal values of $j$ are the same in both families. Hence,
$$\sum_{k<m}x_k - x_m = 0.$$
However, using \eqref{xmzero} to rewrite the first term, we find that $x_m = 0$. Therefore $x_k \geq 0$ for all $k\leq m$, by choice of $x_m$, and from \eqref{xmzero} again we find that $x_k = 0$.

Now consider row $i = 2$ and set $m = a_2$. We wish to show that $x_k = 0$ for all $k \leq m$ except one, for which $x_k = -1$. In this case, \eqref{xmzero} becomes
\begin{equation}
\label{xmminus}
x_1 + \dots + x_{m-1} + x_m = -1.
\end{equation}
Although we will keep the set $\mathcal{S}^\prime$ as defined before, this time we use
$$\mathcal{S} = \left\{K\in\mathcal{M} \left\vert \left<K,[v_2]\right> = -m, D^{p/q}([K])=0\right\}\right. ,$$
so that the maximum \eqref{maxjone} becomes
$$1 + \sum_{k>m} \abs{x_k}.$$
while the second maximum \eqref{maxjtwo} remains unchanged after we have defined $x_m$ to be the smallest of the $x_k$ for $k\leq m$. We claim that the two maxima are equal. Indeed, observe that $(j_{\max},-m,b_3+2,b_4,\dots,b_\ell)\in\mathcal{S}\cap\mathcal{K}^\prime$, and $(j_{\max},m-2,b_3,\dots, b_\ell) \in \mathcal{S}^\prime\cap \mathcal{K}^\prime$. Hence, comparing the maxima, we obtain:
$$\sum_{k<m} x_k - x_m = 1.$$
If we rearrange \eqref{xmminus}, as before, we find that $x_m = -1$. If $m = 2$, then \eqref{xmminus} yields the result. If, on the other hand, $m > 2$, then repeat this process with
$$\mathcal{S}^{\prime\prime} = \left\{ K \in \mathcal{M} \left\vert \left<K,[v_2]\right> = m-4, D^{p/q}([K]) \right\} \right.$$
and $x_{m-1}$ defined to be the next smallest after $x_m$. We find that $x_{m-1}+x_m = -1$, from which $x_{m-1} = 0$. Hence $x_k \geq 0$ for all $k \leq m-1$, and it follows from \eqref{xmminus} that $x_k = 0$ for $k \leq m-1$, as required.

By this point we are finally almost there. What remains to establish is the changemaker condition on $x$. Using the labels $\sigma_i$ established, and defining $\sigma_0$ to be the other unital entry, change signs as usual so that $\sigma_i \geq 0$. Let
$$J:= \left\{ \left<K,[v_1]\right> \left\vert K \in \mathcal{M}, \left<K,[v_2]\right> = -a_2, D^{p/q}([K])=0 \right\} \right. ,$$
and observe that $J$ consists of all values $j \equiv n$ from $2-j_{\max}$ to $j_{\max}$. The asymmetry is a result of the fact that if $K = (j,-a_2,*,\dots,*)$ is a relevant maximiser with appropriate values *, then $K \in \mathcal{K}^\prime$ if $j \geq 1$, whereas $K \sim K^\prime := (j-2,a_2,*,\dots,*)\in \mathcal{K}^\prime$ if $j < 1$. Thus, in light of the evaluation on $[v_2]$,
$$j = \sigma_0 - \sum_{i\geq 1} \alpha_i \sigma_i = 1 - \sum_{i\geq 1} \alpha_i \sigma_i$$
attains these values too. By writing $\alpha_i = -1 + 2\chi_i$ (where $\chi_i \in \{0,1\}$), we obtain
$$j = 1 + \sum_{i \geq 1} \sigma_i - 2\sum_{i\geq 1} \chi_i\sigma_i = j_{\max} - 2\sum_{i\geq 1} \chi_i\sigma_i,$$
and thus $\left\{\left. \sum_{i\geq 1} \chi_i\sigma_i \right\vert \chi_i \in \{0,1\} \right\}$ consists of all integers from $0$ to $\sum_{i\geq 1} \sigma_i$. This is precisely the condition for a changemaker set.

\subsection{Proof of Theorem \ref{changemaker} when $0 < p < q$}

This proof is extremely similar to the previous one, so we only outline the differences. Crucially, $n = 1$, so via Theorem \ref{main} it follows that all the deficiencies $D^{p/q}([K])$ vanish for any maximiser $K$. This fact makes the proof much easier.

To ensure that all non-zero entries in $y_i$ are $\pm1$ for all $i$, we use the maximiser
$$K = (1,b_2,\dots,b_{i-2},-a_{i-1},a_i,b_{i+1},\dots,b_\ell).$$
We know that this choice of $K$ is in fact a maximiser since it is a push-down of
$$(1,-b_2,\dots,-b_{i-2},-b_{i-1},-b_i,b_{i+1},\dots,b_\ell) \in \mathcal{K}.$$
The same argument as above then yields our results.

To show that rows $y_i$ and $y_j$ where $j - i \geq 2$ do not overlap (i.e. share non-zero entries in the same spots), one must be a little more careful. Supposing that $a_i \neq 2$ (i.e. that $a_i > 2$), one uses the maximiser
$$K = (1,b_2,\dots,b_{i-2},-a_{i-1},a_i,-a_{i+1},b_{i+2},\dots,b_{j-2},b_{j-1},a_j,b_{j+1},\dots,b_\ell),$$
which, by pushing-down at $v_1$ to the right and at $v_{j-1}$ to the left, can be obtained from
$$(1,-b_2,\dots,-b_{i-2},-b_{i-1},a_i-4,-b_{i+1},-b_{i-2},\dots,-b_{j-2},a_{j-1},a_j-2,b_{j+1},\dots,b_\ell) \in \mathcal{K}.$$
If instead $a_i = 2$ and there is some $k \in \{2,\dots,i-1\}$ such that $a_k \neq 2$, then we use
$$K = (1,b_2,\dots, b_{i-1},a_i,-a_{i+1},b_{i+2},\dots,b_{j-2},b_{j-1},a_j,b_{j+1},\dots,b_\ell),$$
a push-down of $(1,b_2,\dots,b_{i-1},a_i-2,-b_{i+1},-b_{i+2},\dots,-b_{j-2},a_{j-1},a_j-2,b_{j+1},\dots,b_\ell) \in \mathcal{K}$. Finally, if $a_k = 2$ for all $k = 2,\dots,i$, the fact that $y_i \cdot y_j = 0$ implies that if $y_j$ and $y_i$ overlap, then they overlap in two places. Consequently, since $y_{i-1} \cdot y_i = -1$, it follows that $y_j$ also overlaps with $y_{i-1}$, and hence as $y_{i-1} \cdot y_j = 0$, that $y_j$ overlaps in two places with $y_{i-1}$. Iterating this, we find eventually that $y_j$ and $x$ overlap, violating the condition $x \cdot y_j = 0$, since $x$ contains precisely one non-zero entry (as $x^2 = 1$). Hence, $y_i$ and $y_j$ cannot overlap.

To show that $y_i$ and $y_{i+1}$ have only one overlap (in which they are opposite in sign), one uses
$$K = (1,b_2,\dots,b_{i-2},-a_{i-1},a_i,a_{i+1}-2,b_{i+2},\dots, b_\ell),$$
which is a push-down of $(1,-b_2,\dots,-b_{i-2},-b_{i-1},a_i-2,a_{i+1}-2,b_{i+2},\dots,b_\ell) \in \mathcal{K}$.

Because $x^2 = n = 1$, the rest of the computation is trivial, and the theorem is proved.

\subsection{Proof of Theorem \ref{changemaker} when $q = 1$}

This last proof is even easier than in the previous section. Since none of the rows $y_i$ exist, we need only prove the statement about $x$; in the absence of the other rows, the only adjustments we need make to the proof of the changemaker statement are to define instead
$$J := \left\{ \left<K,[v_1]\right> \left\vert K \in \mathcal{M}, D^{p/q}([K])=0 \right\} \right.,$$
and remove the assumption that $\sigma_0 = 1$. Once this is done, the modified statement follows easily.

\subsection{A Remark on Vanishing Deficiencies}

In its current form, the reader will hopefully have noticed the asymmetry in Theorem \ref{changemaker} concerning the number of deficiencies which vanish. If $n$ is odd, we only require one to vanish, but if $n$ is even, then we require $q - r + 1$. It is possible that by choosing a different function $\mathfrak{r}$ we can remove this asymmetry, but as of the current writing we have been unable to do so.

What we can say, however, is that in the special case when $q = 2$, some simplifications are possible (c.f. \cite{GreeneBraid, OSUnknot}). In practice, the following proposition is most readily applied when $\mathfrak{t}$ is the unique $\Spin$-structure.

\begin{prop}
\label{spinsufficient}
In the case $q = 2$, Theorem \ref{changemaker} applies if we use a weaker assumption on the number of vanishing deficiencies. Namely, if $n$ is even, we require only that
$$d(Y,\mathfrak{t}) - d(S^3_{-p/q}(U),\mathfrak{t}) = 0,$$
for some $\mathfrak{t} \in \Spinc(Y)$.
\end{prop}
\begin{proof}
When $q = 2$, notice that $p/q = [n,2]$. We relabel row $y_2$ as $y$ for convenience, and set $y = (1,1,0,\dots,0)$ without loss of generality. Then $x \cdot y = -1$ tells us that
\begin{equation}
\label{twoxs}
x_1 + x_2 = -1.
\end{equation}
We let $x_2 \leq x_1$, also without loss of generality. Observe that $x_1 \geq 0$, else $x_1 + x_2 \leq -2$.

Now define a set
$$\mathcal{S} = \left\{ K \in \mathcal{M} \left\vert \left<K, [v_2]\right> = 0, D^{p/q}([K]) = 0\right\}\right.,$$
and observe that the maximal value $j_{\max}$ of $\left< K, [v_1]\right>$ obtained by letting $K$ range over $\mathcal{S}$ satisfies $j_{\max} \geq 0$, since we know that at least one deficiency vanishes. If $j_{\max} > 0$, however, then this means that at least $q-r+1$ deficiencies vanish, and Theorem \ref{changemaker} applies. Thus, suppose $j_{\max} = 0$. By arguments similar to those in Section \ref{mainproof}, we find that
$$j_{\max} = x_1 - x_2 + \sum_{i\geq 3} \abs{x_i} = 0,$$
and on substituting from \eqref{twoxs},
$$2x_1 + 1 + \sum_{i\geq 3} \abs{x_i} = 0.$$
Since none of the terms on the LHS are negative, we have a contradiction. Hence $j_{\max} \neq 0$, and Theorem \ref{changemaker} applies.
\end{proof}

\bibliography{symbib}
\bibliographystyle{plain}

\end{document}